\documentclass[12pt]{amsart}
\usepackage{amssymb,verbatim,enumerate,ifthen}
\usepackage{amsmath}
\usepackage[mathscr]{eucal}
\usepackage[utf8]{inputenc}
\usepackage[T1]{fontenc}
\usepackage{cite}
\usepackage[bottom=0.5in,top=1.5in]{geometry}

\newcommand{\Jarnik}{Jarn{\'\i}k}
\usepackage{tikz}
\numberwithin{equation}{section}
\def\eq#1{{\rm(\ref{#1})}}
\def\Eq#1#2{\ifthenelse{\equal{#1}{*}}
  {\begin{equation*}\begin{aligned}[]#2\end{aligned}\end{equation*}}
  {\begin{equation}\begin{aligned}[]\label{#1}#2\end{aligned}\end{equation}}}

\usepackage{xcolor}

\makeatletter
\@namedef{subjclassname@2010}{%
  \textup{2020} Mathematics Subject Classification}
\makeatother

\newtheorem{thm}{Theorem}
\newtheorem{cor}{Corollary}
\newtheorem{prop}{Proposition}
\newtheorem{lem}{Lemma}

\newtheorem{rem}{Remark}

\numberwithin{equation}{section}

\newcommand\nolabel[1]{\nonumber}
\newcommand\x{{\mathbf{x}}}

\newcommand\R{\mathbb{R}}
\renewcommand\P{\mathscr{P}}
\newcommand\N{\mathbb{N}}

\newcommand\calC{\mathcal{C}}
\newcommand\Cts{\mathcal{C}^{2\#}}

\newcommand{\QA}[1]{\mathscr{A}^{[#1]}}

\newcommand{\scrI}{\mathscr{I}}

\DeclareMathOperator{\cl}{cl}
\DeclareMathOperator{\interior}{int}

\DeclareMathOperator{\CM}{\mathcal{CM}}

\title{
On the $L^1$ and pointwise divergence of continuous functions
}

\author{Karol Gryszka}
\author{Pawe\l{} Pasteczka}
\address{Institute of Mathematics, Pedagogical University of Krak\'ow, Podchor\k{a}\.{z}ych str 2, 30-084 Krak\'ow, Poland}
\email{\{karol.gryszka, pawel.pasteczka\}@up.krakow.pl}
\begin{document}
\begin{abstract}
For a family of continuous functions $f_1,f_2,\dots \colon I \to \mathbb{R}$ ($I$ is a fixed interval) with $f_1\le f_2\le \dots$ define a set
 $$ I_f:=\big\{x \in I \colon \lim_{n \to \infty} f_n(x)=+\infty\big\}.$$
We study the properties of the family of all admissible $I_f$-s and the family of all admissible $I_f$-s under the additional assumption
$$ \lim_{n \to \infty} \int_x^y f_n(t)\:dt=+\infty \quad \text{ for all }x,y \in I\text{ with }x<y.$$

The origin of this problem is the limit behaviour of quasiarithmetic means.
\end{abstract}
\maketitle

%

\section{Introduction}

Quasiarihmetic means were introduced in 1920-s/30-s by de Finetti \cite{Def31}, Knopp \cite{Kno28}, Kolmogorov \cite{Kol30} and Nagumo \cite{Nag30}. For a continuous and strictly monotone function $F \colon I \to \R$ (here and below $I$ stands for an arbitrary subinterval of $\R$ and $\CM(I)$ stands for a family of all continuous and strictly monotone functions on $I$) we define the \emph{quasiarithmetic mean} $\QA{F} \colon \bigcup_{n=1}^\infty I^n \to I$ by 
\Eq{*}{
\QA{F}(a):=F^{-1} \bigg( \frac{F(a_1)+\cdots+F(a_n)}n \bigg)\qquad \text{ where }n \in \N \text{ and }a=(a_1,\dots,a_n) \in I^n.
}
A function $F$ is called a \emph{generating function} or a \emph{generator} of $\QA{F}$.

It was Knopp \cite{Kno28} who noticed that for $I=\R_+$, $\pi_p(x):=x^p$ ($p\ne 0$) and $\pi_0(x):=\ln x$, the quasiarithmetic mean $\QA{\pi_p}$ coincides with the $p$-th power mean $\P_p$.

Adapting the classical result $\lim_{p \to +\infty} \P_p = \max$ we say that a family  $(F_n)_{n=1}^\infty$ in $\CM(I)$ is \emph{QA-maximal} provided 
\Eq{*}{
\lim_{n \to\infty}\QA{F_n} = \max \text{ pointwise.}
}

There are few approaches to this property. First, applying some general results by P\'ales \cite{Pal91}, we can establish the general equivalent condition of being QA-maximal (see also \cite{Pas16a}). More precisely, a sequence $(F_n)_{n=1}^\infty$ of elements in $\CM(I)$ is a QA-maximal family if and only if
\Eq{*}{
\lim_{n \rightarrow \infty}\frac{F_n(x)-F_n(y)}{F_n(z)-F_n(y)} =0\qquad \text{ for all }x,\,y,\,z \in I\text{ with }x<y<z.
}
It is a particular case of an analogous result for deviation and quasideviation means -- see the papers by Dar\'oczy 
\cite{Dar71b,Dar72b}, Dar\'oczy--Losonczi \cite{DarLos70}, Dar\'oczy--P\'ales \cite{DarPal82,DarPal83}, and by P\'ales \cite{Pal82a,Pal83b,Pal84a,Pal85a,Pal88a,Pal88d,Pal88e} for detailed study of these families.

It turns out that under the additional assumption that each generator is twice continuously differentiable with nowhere vanishing first derivative -- from now on we denote family of all such generators by $\Cts(I)$ -- we can establish another equivalent conditions. More precisely, by \cite{Pas16a}, a family $(F_n)_{n=1}^\infty$ of elements in $\Cts(I)$ such that $\frac{F_n''}{F_n'}$ is uniformly lower bounded is QA-maximal if and only if
\Eq{*}{
\lim_{n \to \infty} \int_x^y \frac{F_n''(t)}{F_n'(t)}dt=+\infty \qquad \text{ for all }x,y \in I \text{ with }x<y.
}

Let us emphasize that the operator $\frac{F''}{F'}$ plays a key role in a comparability of quasiarithmetic means. More precisely, by Jensen inequality, for all $F,G \in \Cts(I)$ we have
\Eq{*}{
\QA{F}\le \QA{G} \iff \frac{F''}{F'} \le \frac{G''}{G'}.
}
In view of \cite{Pal13} if $F_1,F_2,\dots \in \Cts(I)$ such that $\QA{F_1}\le \QA{F_2} \le \dots$ then, by \cite{Pas13}, the maximal property is connected with the set
$$\scrI_F:=\bigg\{x \in I \colon \lim_{n\to \infty} \frac{F_n''(x)}{F_n'(x)}=+\infty \bigg\}.
$$
Namely, it was proved that if $\scrI_F=I$ then $(F_n)_{n=1}^\infty$ is QA-maximal. Conversely, for every QA-maximal family $(F_n)_{n=1}^\infty$ the set $\scrI_F$ is a dense subset of $I$.

These results were strengthened in \cite{Pas16a}. More precisely, there is proved  that if the intersection of $\scrI_F$ with an arbitrary open subset of $I$ has a positive Lebesgue measure $\lambda$, then $(F_n)_{n=1}^\infty$ is QA-maximal. 
This assumption is somehow the weakest possible, as for every $X \subset I$ such that $\lambda(X \cap J)=0$ for some open subinterval $J$ of $I$ there exists a family $(G_n)_{n=1}^\infty$ which is not QA-maximal, however $X \subset \scrI_G$. On the other hand, it was proved that $\scrI_Z$ could have the Hausdorff dimension zero for some QA-maximal family $(Z_n)_{n=1}^\infty$.

In what follows our aim is to study the relation between the property of being a max-family and the corresponding set $\mathscr{I}$.

Let us emphasize that the same consideration remains valid for QA-minimal and $\min$-families (with a natural definition). As a matter of fact we can reapply all results below to the reflected means -- for detailed study of reflected means we refer the reader to recent development by Chudziak-P\'ales-Pasteczka \cite{ChuPalPas19}, P\'ales-Pasteczka \cite{PalPas18b} and Pasteczka \cite{Pas16b}. 

\subsection{Rephrasing of the problem}
As all conditions above are expressed in terms of the operator $F \mapsto F''/F'$ we are going to elaborate the properties of this operator. To this end, for a sequence of continuous functions $f_1,f_2,\dots \colon I \to \R$ with $f_1\le f_2 \le \dots$ define
\Eq{*}{
I_f:=\big\{x \in I \colon \lim_{n\to \infty} f_n(x)=+\infty \big\}.
}
Furthermore, let $\Omega(I)$ be a family of all possible $I_f$-s. More precisely,
\Eq{*}{
\Omega(I):=\big\{ I_f \colon f_1,f_2,\dots \in \calC(I) \text{ and }f_1\le f_2\le \dots \big\}.
}
Additionally, if
\Eq{E:intinfty}{
\lim_{n \to \infty} \int_x^y f_n(t)dt=+\infty \qquad \text{ for all }x,y \in I \text{ with }x<y,
}
then the family $(f_n)_{n=1}^\infty$ is called \emph{max-family}.
Based on the previous section we obtain the following results.
\begin{prop}[\cite{Pas16a}, Proposition~4.1]\label{Prop:P1}
Let $X\subset I$ be an arbitrary set.
If $\lambda(X \cap J)=0$ for some open interval $J$, then there exists a sequence 
$(f_n \colon I \to \R)_{n=1}^\infty$ of continuous functions with $f_1 \le f_2\le\dots$ which is not a max-family, although $I_f \supset X$.
\end{prop}

\begin{prop}[\cite{Pas16a}, Proposition~4.2]\label{Prop:P2}
Let $(f_n \colon I \to \R)_{n=1}^\infty$ be a family of continuous functions with $f_1 \le f_2\le\dots$ and such that $I_f$ intersected with each open subinterval of $I$ has a positive Lebesgue measure. Then $(f_n)$ is a max-family. 
\end{prop}

Propositions above motivates us to define 
\Eq{*}{
\Omega_0(I):=\big\{ I_f \colon (f_n \colon I \to \R)_{n=1}^\infty \text{ is a max-family} \big\}.
}
Obviously $\Omega_0(I) \subseteq \Omega(I)$. 
The aim of this paper is to show several important facts concerning the set $\Omega(I)$, $\Omega_0(I)$, and their common relations. 

\section{General properties of $\Omega(I)$ and $\Omega_0(I)$}
First, let us present our initial result which shows that $\Omega(I)$ consists of $G_\delta$ sets only.
\begin{lem}
Let $(f_n \colon I \to \R)_{n=1}^\infty$ be a family of continuous function with $f_1\le f_2\le \cdots$. Then $I_f$ is a $G_\delta$ set.
\end{lem}
\begin{proof}
Set $A_{n,m}:=\{x\in I\colon f_n(x)>M\}$ where $m, n\in\mathbb{N}$. Then $A_{n,m}$ are open and
$$x\in I_f\iff x\in \bigcap_{M\in \mathbb{N}}\bigcup_{N\in\mathbb{N}}\bigcap_{n\geq N}A_{n,M}.$$
But the sequence $(f_n)_{n=1}^\infty$ is monotone, thus 
$\bigcap_{n\geq N}A_{n,M}=A_{N,M}$ and $I_f$ is $G_\delta$.
\end{proof}
Let us now prove that the set $\Omega(I)$ is closed under finite union and closed under countable intersection (with additional assumption).
\begin{lem}\label{lem:sum}
Let $I$ be an interval and $(f_n \colon I \to \R)_{n=1}^\infty$ and $(g_n \colon I \to \R)_{n=1}^\infty$ be two families of continuous functions with $f_1 \le f_2\le\dots$ and $g_1 \le g_2\le\dots$. 
Then $I_{f+g}=I_f\cup I_g$.
\end{lem}
\begin{proof}
Observe that as $f_n\ge f_1$ and $g_n\ge g_1$ we obtain $I_f\subseteq I_{f+g}$ and $I_g\subseteq I_{f+g}$.
Therefore $I_f \cup I_g \subseteq I_{f+g}$. 

To prove the converse inclusion take $x \in I_{f+g}$ arbitrarily. Then we have
\Eq{*}{
+\infty=\lim_{n \to \infty}\big( f_n(x)+g_n(x)\big)=\lim_{n \to \infty} f_n(x)+\lim_{n \to \infty}g_n(x),
}
which shows that $\lim_{n \to \infty} f_n(x)=+\infty$ or $\lim_{n \to \infty}g_n(x)=+\infty$, i.e. $x \in I_f \cup I_g$.
\end{proof}

\begin{cor}
Let $I \subset \R$ be an arbitrary interval. Then
\begin{enumerate}[(i)]
    \item for all $J, K \in \Omega(I)$ we have $J \cup K \in \Omega(I)$;
    \item for all $J \in \Omega(I)$ and $K \in \Omega_0(I)$ we have $J\cup K \in \Omega_0(I)$.
\end{enumerate}
\end{cor}

Observe that Lemma \ref{lem:sum} fails to be true for countable sequence of families of continuous functions. To see that set take two sequences of families $\big(f^{(1)}_n\big)_{n=1}^\infty,\big(f^{(2)}_n\big)_{n=1}^\infty,\dots$ and $\big(g^{(1)}_n\big)_{n=1}^\infty,\big(g^{(2)}_n\big)_{n=1}^\infty,\dots$ on $\R$ defined by
\begin{align*}
f_n^{(1)}&\equiv \phantom{-}n, &\qquad \text{ for all }n \in \N;\\
f_n^{(i)}&\equiv -1,&\qquad \text{ for all }n \in \N \text{ and }i\ge2,\\
g_n^{(i)}&\equiv \phantom{-}1,&\qquad \text{ for all }n \in \N \text{ and }i\ge1.
\end{align*}
Then we have
\Eq{*}{
I_{\sum_i f^{(i)}}=\emptyset,\quad\text{ while }\quad \bigcup_iI_{f^{(i)}}=\R.
}
On the other hand
\Eq{*}{
I_{\sum_i g^{(i)}}=\R\quad\text{ and }\quad \bigcup_i I_{g^{(i)}}=\emptyset.
}
Above equalities show that both inclusions may fail.


Let us note as a curiosity the following remark that mimics Lemma \ref{lem:sum}.
\begin{rem}\label{lem:product}
Let $I$ be an interval and $(f_n \colon I \to \R)_{n=1}^\infty$ and $(g_n \colon I \to \R)_{n=1}^\infty$ be two families of continuous functions with $f_1 \le f_2\le\dots$ and $g_1 \le g_2\le\dots$. 
Then $I_{f\cdot g}= I_f\cup I_g$ provided that if $x\in I_f$ ($x\in I_g$), then there is $\eta\in (0,+\infty]$ such that $g_n(x)\to \eta$ ($f_n(x)\to \eta$).
\end{rem}
\begin{proof}
Take $x \in I_{f\cdot g}$ arbitrarily. Then we have
\Eq{*}{
+\infty=\lim_{n \to \infty} \big(f_n(x)\cdot g_n(x)\big)
}
which is possible when $\lim_{n \to \infty} f_n(x)$ and $\lim_{n \to \infty}g_n(x)$ are non-negative and at least one of them is $+\infty$. Hence $x \in I_f \cup I_g$.

To show the converse inclusion assume that $x\in I_f$ and $\lim\limits_{n\to\infty}g_n(x)=\eta>0$. Then,
$$\lim\limits_{n\to\infty}\big(f_n(x)\cdot g_n(x)\big)= \lim\limits_{n\to\infty}f_n(x)\cdot \lim\limits_{n\to\infty}g_n(x)=+\infty\cdot \eta=+\infty$$
and $x\in I_{f\cdot g}$. Thus $I_f\subseteq I_{f\cdot g}$ and similarly, $I_g\subseteq I_{f\cdot g}$.
\end{proof}

\begin{lem}\label{lem:4}
Let $I$ be an interval and $(D_n)_{n=0}^\infty$ be a family of closed subsets of $I$ such that 
$D_0=I$ and $D_{n+1} \subset \interior D_n$ for all $n \in \N_+ \cup\{0\}$.
Then $\bigcap_{n=0}^\infty D_n \in \Omega(I)$.
\end{lem}
\begin{proof}
Indeed, in view of Tietze(-Urysohn-Brouwer) theorem, for every $n \in \N$ there exists a continuous function $\delta_n \colon I \to [0,1]$ such that
\Eq{*}{
 \delta_n(x)=
 \begin{cases}
 0 & \text{ for }x \in I \setminus D_n; \\
 1 & \text{ for }x \in D_{n+1}.
 \end{cases}
}
Define $d_n:=\sum_{i=0}^n \delta_n$. Then for every $x \in \bigcap_{k=0}^\infty D_k=:D_\infty$ we have $d_n(x)=n$. In particular $I_d \supseteq D_\infty$. 
Now fix $N \in \N$. Then as $D_n \subseteq D_{N}$ for $n \ge N$, we have $\delta_n(x)=0$ for $x\in I\setminus D_{N}$. Thus 
\Eq{*}{
d_n(x)<N \qquad \text{ for all }n \ge N \text{ and }x \in I \setminus D_{N}.
}

This proves that $I_d \subseteq D_{N}$. As $N$ was an arbitrary natural number we get $I_d \subseteq D_\infty$, which completes the equality $I_d=\bigcap_{k=0}^\infty D_n$.
Therefore $\bigcap_{k=0}^\infty D_n \in \Omega(I)$.
\end{proof}

\begin{lem}[\cite{Pas16a}, Proposition 4.3]\label{dH0}
For every interval $I$ there exists a max-family $(z_n \colon I \to \R_+)_{n=1}^\infty$ such that $\dim_H I_z=0$. 

In particular $\Omega_0(I)$ contains a set of zero Hausdorff dimension.
\end{lem}

\section{Max-families with noninteger Hausdorff dimension}

In the theory of fractals two natural questions are present: \emph{Is there a set which Hausdorff dimension equals the given number?} And if the answer is positive: \emph{What  additional features this set can have?} The answers to the first question was given for instance by \cite{Gry2019PKM,Soltan2006,Soltan2020}, while the answer to the second question depends on the feature (see \cite{Sq2017} for connections with ergodicity and continued fractions, \cite{falconer_1985} for properties of distance sets and \cite{mauldin_2016} for examples of subrings of $\mathbb{R}$).

In this section we present two construction of max-families with an arbitrary Hausdorff dimension  $\theta \in (0,1)$. In the first (Cantor-type) approach we show that $\Omega_0(I)$ contains a set which can be factorized to a nowhere dense set of Hausdorff dimension $\theta$ and a dense set of Hausdorff dimension zero. 
In the second (\Jarnik-type) approach we construct a set in $I_f \in \Omega_0(I)$ such that 
$\dim_H(I_f\cap U)=\theta$ for an arbitrary open interval $U \subset I$. We provide two constructions; the Cantor-like sets can be described directly, while the second one relies on number-theoretical approach and thus does not give any insight on how does the provided set actually look like.

\subsection{Cantor-type construction}
We recall some basic notation and definitions from the fractal theory. We call a function $f:X\to X$ a \emph{contraction} if it is Lipschitz with constant $c_f\in (-1,1)$. We call a finite set $\mathcal{F}=\{f_1, \ldots, f_n\}$ of contractions defined on a compact metric space $X$ an \emph{iterated function system}, or \emph{IFS}. We say that the IFS $\mathcal{F}$ satisfies the \emph{open set condition} (abbreviated \emph{OSC}) if there exists an open and bounded set $V\neq\emptyset$ with $F(V)\subset V$, where $F(V):=f_1(V)\cup\ldots\cup f_n(V)$ and $f_i(V)\cap f_j(V)=\emptyset$ for all distinct $i,j \in \{1,\dots n\}$. 
\begin{prop}[Moran  \cite{Mor46}]\label{thm:moran}
Suppose that $\mathcal{F}$ satisfies the open set condition and each $f_i\in\mathcal{F}$ is a similarity with contraction constant $c_i$. If $A$ is the compact set such that $F(A)=A$, then $\dim_HA=s$, where $s$ is the unique solution of the equation
$$\sum\limits_{i=1}^n c_i^s=1.$$
\end{prop}
\begin{lem}\label{lem:DnV2}
For every interval $I$ and every $\theta \in (0,1)$ there exists a family $(D_n)_{n=0}^\infty$ of closed sets such that 
\begin{enumerate}[(i)]
 \item \label{Dn1} $D_0=I$,
 \item \label{Dn2} $D_{n+1} \subset \interior D_n$ for all $n \in \N_+ \cup\{0\}$, 
 \item \label{Dn3} $D_\infty:=\bigcap_{n=0}^\infty D_n$ is an invariant set of some IFS,
 \item \label{Dn4} $\dim_H D_{\infty}=\theta$.
\end{enumerate}
\end{lem}
\begin{proof}
One can assume without loss of generality that $I=[0,1]$. Let $m:=(\tfrac12)^{1/\theta} \in (0,\frac12)$.
Take $\varepsilon \in (0,\tfrac1{2m}-1)$ and define an IFS $\mathcal{F}=\{F_L, F_R\}$ on $I$ by
\Eq{*}{
F_L(x):=m \cdot (x+\varepsilon) \qquad F_R(x):=1-F_L(x).
}
Then $F_L$ and $F_R$ are similarities with Lipschitz constants $m$, $F_L(I)=[m\varepsilon,m(1+\varepsilon)]$ and $F_R(I)=[1-m(1+\varepsilon),1-m\varepsilon]$.

But 
\Eq{*}{
\varepsilon<\tfrac1{2m}-1
\iff 2m(1+\varepsilon)<1
\iff
m(1+\varepsilon)<1-m(1+\varepsilon),
}
which yields $F_L(I) \cap F_R(I)=\emptyset$ and thus $\mathcal{F}$ satisfies the OSC with $U=\interior I$. Now define a family $(D_n)_{n=0}^\infty$ by
\Eq{*}{
D_0&:=I\\
D_{n+1}&:=F(D_n) & \text{ for } n \ge 0.
}
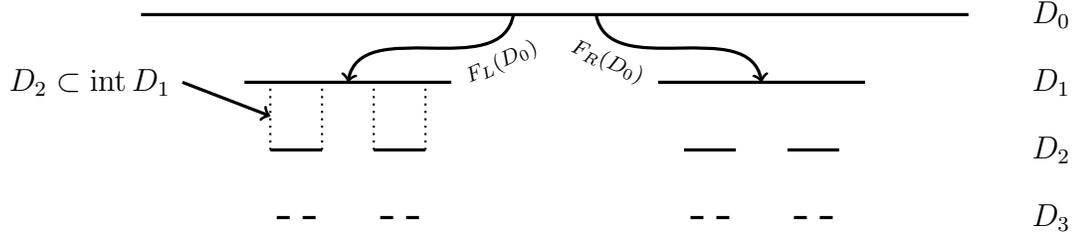
\begin{figure}[h!]
\centering
\begin{tikzpicture}[x=11cm,y=0.9cm, very thick]
\draw (0,0)--(1,0); \draw (1.1,0) node {$D_0$};
\draw (1/8,-1)--(3/8,-1);
\draw (5/8,-1)--(7/8,-1);
\draw (1.1,-1) node {$D_1$};

\foreach \x in{5,9,21,25}{\draw (\x/32,-2)--++(2/32,0);}
\draw (1.1,-2) node {$D_2$};

\foreach \x in{21,25,37,41,85,89,101,105}{\draw (\x/128,-3)--++(2/128,0);}
\draw (1.1,-3) node {$D_3$};

\foreach \x in {5,7,9,11}{
\draw[dotted,thick] (\x/32,-2)--++(0,1);}

\draw[->] (0.05,-1) node[left]{$D_2\subset \interior D_1$} --(20/128,-1.5);
\draw[->] (0.45,0) to [out=260, in=80, edge node={node [sloped,below,pos=0.2] {{\tiny $F_L(D_0)$}}}] (1/4,-1);
\draw[->] (0.55,0) to [out=280, in=100, edge node={node [sloped,below,pos=0.2] {{\tiny $F_R(D_0)$}}}] (3/4,-1);
\end{tikzpicture}
\caption{Construction of $D_n$'s.}\label{figure1}
\end{figure}
Obviously each $D_n$ is a closed subset of $I$. Now we prove that the sequence $(D_n)_{n=0}^\infty$ satisfies all conditions \eq{Dn1}--\eq{Dn4}.
Condition \eq{Dn1} is obvious. To show the second property observe that
\Eq{*}{
D_1 =\big[m\varepsilon,m(1+\varepsilon)\big] \cup \big[1-m(1+\varepsilon),1-m\varepsilon\big] \subset (0,1)=\interior D_0.
}
Moreover if $D_{n+1} \subset \interior D_n$ for some $n \in \N$, then as both $F_L$ and $F_R$ are homeomorphisms we get
\Eq{*}{
F_L(D_{n+1}) \subset \interior F_L(D_n) \quad \text{ and } \quad F_R(D_{n+1}) \subset\interior F_R(D_n),
}
and thus $D_{n+2} \subset \interior D_{n+1}$. By simple induction we obtain property \eq{Dn2}.

Denote $D_\infty=\bigcap_{n=0}^\infty D_n$. Condition \eq{Dn3} follows from \eq{Dn2} and the general theory of fractal sets (see for instance Theorem 9.1 in \cite{Fal14}).

To check the last condition denote $s:=\dim_H D_\infty$. Then, by Proposition \ref{thm:moran} we have
$2m^s=1$. Thus $s=\frac{\ln (1/2)}{\ln m}=\theta$, which is \eq{Dn4}.
\end{proof} 


\begin{rem} In Lemma \ref{lem:DnV2} the set $D_\infty$ has Lebesgue measure zero. It turns out that the entire construction presented there is equivalent to the construction of the uniform Cantor set on the interval $[m\varepsilon/(1-m),1-m\varepsilon/(1-m)]$ with the middle part of length $\frac{(1-2m)(1-m-2m\varepsilon)}{1-m}$ removed  (that is both constructions lead to the same set).
\end{rem}



\begin{lem}\label{lem:theta}
For every $\theta \in [0,1]$ there exists a family $(d_n \colon I \to \R)_{n=1}^\infty$ of continuous functions such that $0\le d_1\le d_2 \le \dots $, the set $I_d$ is nowhere dense, and $\dim_HI_d=\theta$.
\end{lem}
\begin{proof}
For $\theta =0$ the statement is an easy implication of Lemma~\ref{dH0}. Similarly, for $\theta=1$ we can take $d_n \equiv n$.

For $\theta \in (0,1)$  set a family $(D_n)_{n=0}^\infty$ like in Lemma~\ref{lem:DnV2}. By Lemma~\ref{lem:4} we get that 
$I_d=\bigcap_{n=0}^\infty D_n \in \Omega(I)$ which is easily equivalent to our statement.
\end{proof}

\begin{thm}\label{thm:anydH}
For every $\theta \in [0,1]$, there exists a max-family $(f_n \colon I \to \R)_{n=1}^\infty$ such that $\dim_HI_f=\theta$ and $I_f$ can be decomposed to a nowhere dense set and a set of Hausdorff dimension zero.
\end{thm}
\begin{proof}
Take $\theta \in [0,1]$ arbitrarily. By Lemma~\ref{dH0} one can take a max-family $(z_n \colon I \to \R_+)_{n=1}^\infty$ such that $\dim_HI_z=0$. Furthermore, by Lemma~\ref{lem:theta} we can take a family $(d_n \colon I \to \R_+)_{n=1}^\infty$ such that $I_d$ is nowhere dense, $\dim_HI_d=\theta$, and
\Eq{dmon}{
0\le d_1\le d_2\le \cdots.
}

Let $f_n:=d_n+z_n$ for all $n \in \N$. Then, by Lemma~\ref{lem:sum}, we have $\dim_HI_q=\theta$ and $I_f=I_d\cup I_z$ admit a decomposition mentioned in the statement.

Therefore it is sufficient to show that $(f_n)_{n=1}^\infty$ is a max-family. As $(z_n)_{n=1}^\infty$ is a max-family we have $0\le z_1\le z_2 \le \dots$. Binding this property with \eq{dmon} and the definition of the sequence $(f_n)_{n=1}^\infty$ we easily obtain
\Eq{*}{
0 \le f_1 \le f_2 \le \cdots.
}
Thus the only remaining part to be proved is that \eq{E:intinfty} holds. However, as $d_n \ge 0$ for all $n \in \N$ and $(z_n)_{n=1}^\infty$ is a max-family, we obtain
\begin{multline*}
\lim_{n\to\infty} \int_x^y f_n(t)dt=\lim_{n\to\infty} \int_x^y (d_n+z_n)(t)dt \ge \lim_{n\to\infty} \int_x^y z_n(t)dt=+\infty \\ \text{ for all }x,y \in I\text{ with }x<y,
\end{multline*}
which completes the proof.
\end{proof}

\begin{rem} 
The set $I_q$ in Theorem \ref{thm:anydH} cannot have positive measure (compare with Proposition \ref{Prop:P1} and Proposition \ref{Prop:P2}). It turns out that the set $D_\infty$ obtained in Lemma \ref{lem:DnV2} has zero Lebesgue measure. In fact all Borel sets $D$ with $\dim_HD\in(0,1)$ have zero measure. This is because the Hausdorff measure $H^d$ is up to a constant equivalent to the Lebesgue one, i.e. for integer $d$, $H^d(D)=c(d)\lambda^d(D)$ where $c(d)$ is a known constant. Hence, if $\lambda^1(D)>0$, then $H^1(D)>0$ and thus $\dim_HD=1$.
\end{rem}

\subsection{\Jarnik-type construction}
In what follows we show that for every $\theta \in (0,1)$ there exists a max-family $(f_n \colon I \to \R)_{n=1}^\infty$ such that $\dim_H(I_f \cap J)=\theta$ for every open subinterval $J \subset I$.
The important theorem concerning neighbourhoods of rational numbers will be used. 

\begin{prop}[\Jarnik{} \cite{Jar31}]\label{prop:Jarnik}
Suppose $\alpha>2$. Let $Q_\alpha$ be the set of real numbers $x\in [0,1]$ for which the inequality
$$\|qx\|\leq q^{1-\alpha}$$
is satisfied by infinitely many positive integers $q$, where 
$$\|y\|:=\min\limits_{z\in\mathbb{Z}}|y-z|.$$
Then $\dim_H Q_\alpha=2/\alpha$. Moreover, $Q_\alpha$ is dense in $[0,1]$ and $\dim_H(Q_\alpha\cap J)=2/\alpha$ for every subinterval $J\subset [0,1]$.
\end{prop}

\begin{thm}
For every interval $I$ and every $\theta \in [0,1]$ there exists a max-family $(f_n \colon I \to \R)_{n=1}^\infty$ of continuous functions such that $\dim_H(I_f\cap U)=\theta$ for every open subinterval $U$ of $I$.
\end{thm}

\begin{proof}
 Set $I=[0,1]$ and $\alpha_0:=\tfrac2\theta$. For $q \in \N$ and $\alpha>2$ define
\Eq{*}{
Y_{q,\alpha}&:=\{x \in I \colon \|qx\| \le q^{1-\alpha}\},\\
Z_{q,\alpha}&:=\{x \in I \colon \|qx\| \le \tfrac{q+1}{q}q^{1-\alpha}\}.
}
Then, as $\|\cdot\|$ is continuous, we have $\cl Y_{q,\alpha} \subset \interior Z_{q,\alpha}$ for all $q \in \N$ and $\alpha>\alpha_0$. Now let $\delta_q \colon I \to [0,1]$ be defined by 
\Eq{*}{
 \delta_q(x)=
 \begin{cases}
 0 & \text{ for }x \in I \setminus Z_{q,\alpha_0}; \\
 1 & \text{ for }x \in Y_{q,\alpha_0}.
 \end{cases}
}
Define $r_n \colon I \to \R_+$ by $r_n:=\delta_1+\dots+\delta_n$. 
Then we obtain
\Eq{*}{
I_r\supseteq\{x \in I \colon x\in Y_{q,\alpha_0} \text{ for infinitely many }q\in \N\}=Q_{\alpha_0}.
}
On the other hand for all $\alpha<\alpha_0$ there exists a number $q_0 \in \N$ such that 
\Eq{*}{
\tfrac{q+1}{q} q^{1-\alpha_0}\le q^{1-\alpha} \qquad \text{ for all }q\ge q_0.
}
Thus, for all $\alpha\in(2,\alpha_0)$ we have 
\Eq{*}{
I_r&\subseteq\{x \in I \colon x\in Z_{q,\alpha_0} \text{ for infinitely many }q\in \N\}\\
&=\{x \in I \colon x\in Z_{q,\alpha_0} \text{ for infinitely many }q\in \N\text{ with }q \ge q_0\}\\
&\subseteq \{x \in I \colon x\in Y_{q,\alpha} \text{ for infinitely many }q\in \N\text{ with }q \ge q_0\}\\
&= \{x \in I \colon x\in Y_{q,\alpha} \text{ for infinitely many }q\in \N\}=Q_\alpha.
}
Finally we have
\Eq{*}{
Q_{\alpha_0} \subseteq I_r \subseteq \bigcap_{\alpha \in (2,\alpha_0)} Q_\alpha.
}
Therefore by Proposition~\ref{prop:Jarnik} we obtain that for every open interval $U \subset I$ we have 
\Eq{*}{
\tfrac{2}{\alpha_0}=\dim_H (Q_{\alpha_0} \cap U)&\leq \dim_H(I_r \cap U)
\le \inf_{\alpha\in(2,\alpha_0)} \dim_H( Q_\alpha\cap U)
=\inf_{\alpha\in(2,\alpha_0)}\tfrac{2}{\alpha}=\tfrac{2}{\alpha_0}.
}
Consequently, $\dim_H(I_r \cap U)=\tfrac{2}{\alpha_0}=\theta$ for every open subinterval $U \subset I$. 

Now let $(z_n)_{n=1}^\infty$ be a family from Lemma~\ref{dH0} and let $f_n:=z_n+r_n$ for $n \in \N$. Then, as $r_n \ge 0$, we obtain that $(f_n)_{n=1}^\infty$ is a max-family. Furthermore by Lemma~\ref{lem:sum} we get that $I_f \supset I_r$ and $I_f \setminus I_r$ is of Hausdorff dimension zero which completes the proof. 
\end{proof}

\subsection*{Final conclusions and remarks}
At the very end let us put the reader's attention to few important problems. First, we cannot exclude that $\Omega(I)$ is a family of all $G_\delta$ subsets of $I$ and/or $\Omega_0(I)$ contains all dense $G_\delta$ subsets of $I$. In particular, it is interesting to find a full characterization of all elements of sets $\Omega(I)$ and $\Omega_0(I)$ (our results show that they can be complicated from the measure-theoretical point of view). Second, this problem has a natural multidimensional generalization where the domain of the integral in \eq{E:intinfty} is taken over all open subsets of a given domain. Finally, it is not known if the assumption $f_1\le f_2 \le \dots$ in the definition of $\Omega(I)$ can be relaxed.

\end{document}